\documentclass[11p,reqno]{amsart}
\usepackage{amsmath}
\usepackage{amsfonts}
\usepackage{amssymb}
\usepackage{graphicx}
\usepackage{color}
\newtheorem{theorem}{Theorem}[section]
\newtheorem*{theorem*}{Theorem}
\newtheorem{lemma}{Lemma}[section]
\newtheorem{corollary}[theorem]{Corollary}

\def\aint{\frac{\ \ }{\ \ }{\hskip -0.4cm}\int}

\numberwithin{equation}{section}

\begin{document}
	\title[Holomorphic mappings]{Liouville  theorems and a Schwarz Lemma for holomorphic mappings between K\"ahler manifolds}

\author{Lei Ni}\thanks{The research is partially supported by  ``Capacity Building for Sci-Tech Innovation-Fundamental Research Funds".  }
\address{Lei Ni. Department of Mathematics, University of California, San Diego, La Jolla, CA 92093, USA}
\email{lni@math.ucsd.edu}


\begin{abstract} We derive some consequences of the Liouville theorem for plurisubharmonic functions of L.-F.  Tam and the author. The first result provides a nonlinear version of the complex splitting theorem (which splits off a factor of $\mathbb{C}$ isometrically from the simply-connected K\"ahler manifold with nonnegative bisectional curvature and  a linear growth holomorphic function) of L.-F. Tam and the author. The second set of results concerns the so-called $k$-hyperbolicity and its connection with the negativity of the $k$-scalar curvature (when $k=1$ they are the negativity of holomorphic sectional curvature and Kobayashi hyperbolicity) introduced recently in \cite{Ni-Zheng2} by F. Zheng and the author. We lastly  prove a new Schwarz Lemma type estimate in terms of {\it  only the  holomorphic sectional curvatures of both domain and target manifolds}.
\end{abstract}

\maketitle

\section{Introduction}

 The first goal of this paper is to derive some consequences of the Liouville theorem proved in \cite{NT1}, which asserts that {\it any continuous plurisubharmonic function $u(x)$ defined on a K\"ahler manifold with nonnegative bisectional curvature satisfying that $u(x)=o(\log(r(x)))$, where $r(x)$ denotes the distance function to a fixed point $p$, must be a constant.} We should mention that this result was later generalized in \cite{Liu} and \cite{Ni-Niu2} with weaker assumptions on the curvature. These extension are however not needed for the discussion here.

\begin{theorem}\label{thm1} Let $M^m$ and  $N^n$ be  two complete K\"ahler manifolds.  Assume that the bisectional curvature of $M$ is nonnegative and the bisectional curvature of $N$ is non-positive. (i) Then any holomorphic map $f: M\to N $, whose differential $df$ satisfies
\begin{equation}\label{eq:11}
\limsup_{r(x)\to \infty} \frac{\|df\|(x)}{r^\epsilon (x)}=0
\end{equation}
for any $\epsilon>0$, where $r(x)$ is the distance of $x$ to a fixed point $p\in M$, is totally geodesic.
(ii) If $m=1$, then the same result holds under a weaker assumption that $N$ has nonpositive holomorphic sectional curvature.
\end{theorem}

The result here can be viewed as the nonlinear version of the holomorphic splitting theorem proved in \cite{NT1}, which asserts that {\it on a simply-connected $M$ with nonnegative bisectional curvature, any nonconstant linear growth holomorphic function splits off a $\mathbb{C}$,} since a holomorphic function  can be viewed as the holomorphic map into $\mathbb{C}$, and being totally geodesic implies that $\nabla f$ is parallel, hence the splitting. The linear growth of a holomorphic function implies the boundedness of its gradient by \cite{CY}.

A  result with the same conclusion was  proved earlier for harmonic maps from compact quotients of symmetric irreducible spaces into manifolds with nonpositive complex sectional curvature, namely the celebrated  geometric superrigidity  in \cite{MSYe} (see also \cite{Jost-Yau}). But the result here is  different in nature from the result of Mok-Siu-Yeung in the sense that while allowing the domain manifold being noncompact we imposed a curvature condition on the domain instead. The result of \cite{MSYe} concerns with harmonic maps, which is a considerably larger class than holomorphic maps. On the other hand, \cite{MSYe} is restricted for domain manifolds being compact quotients of symmetric spaces, and at the same time  poses a stronger curvature condition  on the target than the result above.

Applying Cheng's gradient estimate for harmonic maps into the Cartan-Hadamard manifold \cite{Cheng} we have the following corollary.

\begin{corollary}\label{coro1} Let $M^m$ and  $N^n$ be  two complete K\"ahler manifolds.  Assume that the bisectional curvature of $M$ is nonnegative, and  $N$ is a Cartan-Hadamard K\"ahler manifold. Then any holomorphic map $f: M\to N $, whose differential $df$ satisfies
$$
\limsup_{r(x)\to \infty} \frac{d^N(p', f(x))}{r^{1+\epsilon} (x)}=0
$$
for any $\epsilon>0$, where $r(x)$ is the distance of $x$ to a fixed point $p\in M$, and $d^N(p', \cdot)$ is the distance function of $N$ to a point $p'\in N$, is totally geodesic.
\end{corollary}

The proof utilizes a $\partial\bar{\partial}$-Bochner formula for holomorphic maps, which implies the plurisubharmonicity of $\log (A+\|\partial f\|^2)$ (for any $A>0$).

Various concepts of hyperbolicity arise in conjunction with the Schwarz lemma \cite{Kobayashi-H}. Applying a $\partial\bar{\partial}$-lemma (which is collected in Section 2) on the logarithmic of $k$-dimensional volume   we derive,  in Section 3, results related to {\it the $k$-hyperbolicity of a K\"ahler manifold} in conjunction with the so-called $k$-th scalar curvature. Below we shall recall and define these concepts after proper motivations.

Recall that in \cite{Ni-Zheng2}, $(N^n, h)$ is defined to have negative (positive) $k$-scalar curvature if
$$
S_k(y, \Sigma)=\frac{k(k+1)}{2Vol(\mathbb{S}^{2k-1})}\int_{|Z|=1, Z\in \Sigma} H(Z)\, d\theta(Z)<0\quad (>0)
$$
for any $y\in N$ and any $k$-dimensional subspaces $\Sigma\subset T_y'N$. Here $H$ denotes the holomorphic sectional curvature of $N$. Namely $H(Z)=R^N(Z, \overline{Z}, Z, \overline{Z})$.  We say $S_k(y)<0$ if $S_k(y, \Sigma)<0$ for every $k$-dimensional $\Sigma$. $N$ is called with negative $k$-scalar curvature if $S_k(y)<0$ everywhere. Regarding compact K\"ahler manifolds with $S_k<0$, an interesting question is {\it  when a compact K\"ahler manifold with $S_2<0$ is  projective}?

The celebrated Brody criterion \cite{Ko3, Kobayashi-H} asserts that $N^n$ is Kobayashi hyperbolic if and only if  any holomorphic map from complex plane $\mathbb{C}$ into $N^n$ must be a  constant map.   Motivated by this criterion of the Kobayashi hyperbolicity (which amounts to $1$-hyperbolicity as illustrated below) and  the work of \cite{Eis}, \cite{Kobayashi-H} (see also \cite{Yau-ma} for the extension to the meromorphic mappings and another definition of an intrinsic $k$-measure) one may {\it  define a compact K\"ahler manifold $N^n$ to be $k$-hyperbolic if and only if any holomorphic map $f:\mathbb{C}^k\to N$ must be degenerate (namely the image of $f$ must be of dimension less than $k$)}.  This provides a natural generalization of Kobayashi hyperbolicity (namely $1$-hyperbolicity), and is equivalent to that a similar pseudo norm on $k$-subspaces of $T_x N$ defined via mapping from $\mathbb{D}^k\to N$ being a norm (see Appendix for a similar proof of Brody's theorem).

  At the mean time recall that the classical Schwarz lemma of Yau-Royden \cite{Yau-sch, Roy} for holomorphic maps from Riemann surfaces into compact K\"ahler manifolds with negative holomorphic sectional curvature implies, via the above Brody's criterion, that any compact K\"ahler manifold $N^n$ with negative holomorphic sectional curvature must be $1$-hyperbolic. In view of that $S_k$  defined above coincides with the holomorphic sectional curvature $H$ for $k=1$  it is hence natural to ask the question (Q): {\it  whether or not any holomorphic map from $\mathbb{C}^k$ into a compact $(N^n, h)$ with $S_k<0$ must be degenerate.} Namely {\it whether or not any compact K\"ahler manifold $N^n$ with  $S_k<0$ is $k$-hyperbolic} (in the sense defined above). The following result provides a strong indication of a positive answer to the question (Q), by showing the positive answer if the map is from a compact quotient of $\mathbb{C}^k$ or $(N^n, h)$ has $Ric_k<0$ (see below for definition). Noting that $Ric_k$ also coincides with $H$ for $k=1$, hence part (ii) provides a generalization of the above mentioned consequence of Royden-Yau.

\begin{theorem}\label{thm2} Assume that $\dim_{\mathbb{C}}M=m \le n=\dim_{\mathbb{C}}N$. (i) Let $(M, g)$ be a compact K\"ahler manifold such that $Ric^M \ge 0$. Let $(N^n, h)$ be a complete K\"ahler manifold such that $S^N_m(y)<0$. Then any holomorphic map $f:M\to N$  must be degenerate. The same result holds if $Ric^M>0$ and $S_m^N\le 0$.

(ii)  Let  $(M^m, g)$ be  noncompact complete K\"ahler manifold with nonnegative scalar curvature and $Ric^M$ is bounded from below. Assume that $(N^n, h)$  has the $m$-Ricci curvature $Ric^N_m \le -\kappa < 0$ (which holds if $Ric^N_m<0$ and $N$ is compact). Then any holomorphic map $f:M\to N$  must be degenerate. In particular, $(N^n, h)$ is $k$-hyperbolic if $Ric_k^N<0$.

(iii) Let  $(M^m, g)$ be  noncompact complete K\"ahler manifold with $Ric^M \ge 0$. Assume that $(N^n, h)$ is noncompact and  has the $m$-Ricci curvature $Ric^N_m <0$, and $f:M\to N$ is a holomorphic map. If $D(x)=\frac{(f^* \omega_h)^m(x)}{\omega_g^m(x)}$ (with $\omega_g$ and $\omega_h$ being the K\"ahler forms of $M$ and $N$) satisfies that
$$
\limsup_{x\to \infty} \frac{D(x)}{r^{\epsilon}(x)}=0
$$
for any $\epsilon >0$, then $f$ must be degenerate.
\end{theorem}
 We define $Ric(x, \Sigma)$ as the Ricci curvature of the curvature tensor restricted to the $k$-dimensional subspace $\Sigma\subset T_x'M$. Namely $Ric(x, \Sigma)(v,\bar{v})=\sum_{i=1}^k R(E_i,\overline{E}_i, v,\bar{v})$ with $\{E_i\}$ being a unitary basis of $\Sigma$.  We say that $Ric_k(x)<0$ if $Ric(x, \Sigma)<0$  for every k-dimensional subspace $\Sigma$. Clearly $Ric_k(x)<0$ implies that $S_k(x)<0$, and  it coincides with $H$ when $k=1$, with $Ric$ when $k=n$. Note that $Ric_k$ is independent for different $k$, unlike $\{S_k\}$ which is decreasing in $k$.   For  $k=1$, the answer has been known to be positive for both $Ric_k>0$ and $RIc_k<0$ (cf. \cite{Wu-Yau}, \cite{Yang}, \cite{Tosatti-Yang}). The result of \cite{Ni-Zheng2} shows that for $k=2$, $Ric_k>0$ does imply the projectivity.

 We should remark that even for the case of the equal-dimension (namely  $n=m$), the result of part (i) of above theorem  seems new (at least the author is not aware of any such a statement in the literatures).  Note that part (i) can be applied to $m$-dimensional tori. Hence if the map could be lift to one from $\mathbb{C}^m$, it would provide a positive answer to the question (Q). Part (ii) is known for equal dimensional case \cite{Ko3}. The part (ii) of above theorem in particular implies that {\it compact K\"ahler manifold $(N^n, h)$ with $Ric^N_k<0$} is $k$-hyperbolic.

For part (iii) of Theorem \ref{thm2}, clearly the negativity of $Ric^N_m$ is needed, since  there are non-degenerate linear maps with bounded $D$ between complex Euclidean spaces.
In general it is still unknown whether or not  the Liouville theorem for the plurisubharmonic functions  holds on a K\"ahler manifold $M^m$ with nonnegative Ricci curvature. Note that   in \cite{Liu} the Liouville theorem for plurisubharmonic functions was proved for K\"ahler manifolds with nonnegative holomorphic sectional curvature,  and in \cite{Ni-Niu2} it was proved for K\"ahler manifolds with nonnegative orthogonal bisectional curvature. On the other hand  under the nonnegativity of Ricci curvature there is a partial result proved in \cite{Ni-JDG} which asserts that {\it  for any plurisubharmonic function $u(x)$ with $o(\log (r(x)))$ growth,  $(\sqrt{-1}\partial \bar{\partial} u)^m\equiv 0$} holds. The part (iii) of the above theorem uses this statement.

In Section 4 we also prove some extension of above result which in particular implies that $f:\mathbb{C}^m\to N$ is asymptotically degenerate if $S_m^N<0$, and it is degenerate  if $\sigma_{m-1}(\partial f)$  is of $o(r^2(x))$. Here $\sigma_{m-1}(\partial f)$ is the $(m-1)$-th symmetric function of the singular values of $\partial f: T_x'M \to T'_{f(x)}N$.

There are many generalizations of the classical Schwarz Lemma on holomorphic maps between unit balls via the work of  Ahlfors, Royden, Lu, Yau, Chen-Cheng-Look, Mok-Yau, etc (see \cite{Kobayashi-H} and \cite{Zheng-B} and references therein). In Section 5 we prove a new version which only involves the holomorphic sectional curvature of domain and target manifolds. Hence it is perhaps the most natural high dimensional generalization of the classical result of Ahlfors. To state the result we introduce the following: For the tangent map $\partial f: T_x'M \to T'_{f(x)}N$ we define its maximum norm square to be
$$
\|\partial f\|^2_m\doteqdot \sup_{v\ne 0}\frac{|\partial f(v)|^2}{|v|^2}.
$$

\begin{theorem}\label{thm:sch} Let $(M, g)$ be a complete K\"ahler manifold such that the holomorphic sectional curvature $H^M(X)/|X|^4 \ge -K$ for some $K\ge0$. Let $(N^n, h)$ be a  K\"ahler manifold such that $H^N(Y)<-\kappa |Y|^4$ for some $\kappa>0$.  Let $f:M\to N$ be  a holomorphic map. Then
\begin{equation}\label{eq:15}
\|\partial f\|^2_m \le \frac{K}{\kappa},
\end{equation}
provided that the bisectional curvature of $M$ is bounded from below. In particular, if $K=0$, any holomorphic map $f: M\to N$ must be a constant map.
\end{theorem}

The proof  uses   a viscosity consideration from PDE theory (cf. \cite{Ni-Wang} for another such application concerning the isoperimetric inequalities in Riemannian geometry). It is also reminiscent of Pogorelov's Lemma (cf. Lemma 4.1.1 of \cite{Gu}) for Monge-Amp\`ere equation, since the maximum eigenvalue of $\nabla^2 u$ is the $\|\cdot\|_m$ for the normal map $\nabla u$ for any smooth  $u$.

 A consequence of Theorem \ref{thm:sch} asserts that {\it the equivalence of the negativities of the holomorphic sectional curvature implies the equivalence of the metrics}. Namely if $M^m$ admits two K\"ahler metrics $g_1$ and $g_2$ satisfying that
$$
-L_1|X|_{g_1}^4\le H_{g_1}(X)\le -U_1|X|_{g_1}^4, \quad -L_2|X|_{g_2}^4 \le H_{g_2}(X)\le -U_2|X|_{g_2}^4
$$
then for any $v\in T_x'M$ we have the estimates:
$$
|v|^2_{g_2}\le \frac{L_1}{U_2}|v|^2_{g_1};\quad |v|^2_{g_1}\le  \frac{L_2}{U_1}|v|^2_{g_2}.
$$
Note that in this case the bisectional curvature lower bound can be easily checked via the  polarization formula, and the result can be stated locally given that the global result is derived from a local estimate. In view of the ample applications of classical Schwarz Lemma we expect further implications of Theorem \ref{thm:sch}.

There are estimates, sometimes even local estimates, associated with Theorems \ref{thm1}, \ref{thm2}, \ref{thm:sch}, \ref{thm41}. These are collected in Corollaries \ref{coro:31},  \ref{Coro-31}, \ref{Coro:42} and estimates (\ref{eq:42}), (\ref{eq:44}), (\ref{eq:53}). These estimates can be conveniently applied to local settings.

We also obtained three statements (namely Corollaries \ref{coro-positive-2} and  \ref{coro-hoop2}) asserting the amount of ``energy" (in terms of the curvature ratio) is needed to have non-degenerate or nonconstant holomorphic maps between two K\"ahler manifolds with curvature constrains. These are in some sense dual versions of the Schwarz Lemma for positively curved manifolds.

\section{$\partial\bar{\partial}$-Bochner formulae for holomorphic mappings}

Let $f: M^m\to N^n$ be a holomorphic map between K\"ahler manifolds. Choose holomorphic normal coordinate $(z_1, z_2, \cdots, z_m)$ near a point $p$ on the domain manifold $M$,  correspondingly  $(w_1, w_2, \cdots, w_n)$ near $f(p)$ in the target. Let $\omega_g=\sqrt{-1}g_{a\bar{\beta}}dz^\alpha\wedge d\bar{z}^{\beta}$ and $\omega_h=\sqrt{-1}h_{i\bar{j}}dw^i\wedge d\bar{w}^{j}$ be the K\"ahler forms  of $M$ and $N$ respectively. Correspondingly, the Christoffel symbols are given
$$
^M\Gamma_{\alpha \gamma}^\beta =\frac{\partial g_{\alpha \bar{\delta}}}{\partial z^{\gamma}}g^{\bar{\delta}\beta}=\Gamma_{\gamma \alpha }^\beta; \quad \quad  ^N\Gamma_{i k}^j =
\frac{\partial h_{i \bar{l}}}{\partial w^{k}}h^{\bar{l}k}=\Gamma_{k i }^j.
$$
We always uses Einstein's convention when there is an repeated index. The symmetry in the Christoffel symbols is due to K\"ahlerity. If the appearance of the indices can distinguish the manifolds we omit the superscripts $^M$ and $^N$.  Correspondingly the curvatures are given by
$$
^MR^\beta_{\alpha \bar{\delta} \gamma}=-\frac{\partial}{\partial \bar{z}^{\delta}} \Gamma_{\alpha \gamma}^\beta; \quad \quad \quad  \,^NR^j_{i \bar{l} k}=-\frac{\partial}{\partial \bar{w}^{l}} \Gamma_{i k}^j.
$$
At the points $p$ and $f(p)$, where the normal coordinates are centered we have that
$$
R_{\bar{\beta}\alpha \bar{\delta} \gamma}=-\frac{\partial^2 g_{\bar{\beta}\alpha}}{\partial z^{\gamma}\partial \bar{z}^{\delta}}; \quad \quad R_{\bar{j}i \bar{l} k}=-\frac{\partial^2 h_{\bar{j}i}}{\partial w^{k}\partial \bar{w}^{l}}.
$$
For a smooth map $df(\frac{\partial\, \,  }{\partial z^\alpha})$ can be written as $\frac{\partial w^i}{\partial z^\alpha} \frac{\partial}{\partial w^i}+\frac{\partial \bar{w}^i}{\partial z^\alpha} \frac{\partial }{\partial \bar{w}^i}$. But for holomorphic map $df(\frac{\partial\, \, }{\partial z^\alpha})=\partial f(\frac{\partial\, \, }{\partial z^\alpha})=\frac{\partial w^i}{\partial z^\alpha} \frac{\partial\quad }{\partial w^i}$, which we also write as $\frac{\partial f^i}{\partial z^\alpha} \frac{\partial\quad }{\partial w^i}$ or $f^i_\alpha \frac{\partial\quad }{\partial w^i}$. Similarly $df(\frac{\partial\, \, }{\partial \bar{z}^\alpha})=\bar{\partial} f(\frac{\partial\, \, }{\partial \bar{z}^\alpha})$. Recall the Hessian of the map is $Ddf(X, Y)=D_Y (df(X))-df(\nabla_Y X)$. For holomorphic map $f$ the information is all in $Ddf(\frac{\partial}{\partial z^\alpha}, \frac{\partial}{\partial z^\beta})=\sum f^i_{\alpha, \beta}dz^\alpha \otimes dz^\beta \otimes \frac{\partial}{\partial w^i}$. Sometimes we also denote as  $D_{X}d_Y f$. In local coordinates
$$
f^i_{\alpha, \beta}=\frac{\partial^2 f^i}{\partial z^\alpha \partial z^\beta}-\Gamma^{\gamma}_{\alpha \beta}\frac{\partial f^i}{\partial z^\gamma}+\Gamma^i_{jk}\frac{\partial f^k}{\partial z^\alpha}\frac{\partial f^l}{\partial z^\beta}.
$$

\begin{lemma}\label{lemma1-1} For a holomorphic map $f: M \to N$,
\begin{eqnarray}\label{Boch1}
\langle \sqrt{-1}\partial \bar{\partial} \|\partial f\|^2, \frac{1}{\sqrt{-1}}v\wedge \bar{v}\rangle &=&\|D_v \partial_{(\cdot)} f\|^2 -\sum_{\alpha, \beta =1}^m g^{\alpha\bar{\beta}}  R^N(\partial f_\alpha, \bar{\partial}f_{\bar{\beta}}, \partial f(v), \bar{\partial} f(\bar{v}))  \nonumber\\
&\quad&  + \sum_{\alpha\beta}g^{\alpha \bar{\beta}}\langle \partial f(R^M_{v\bar{v}})_\alpha,  \bar{\partial}f_{\bar{\beta}}\rangle.
\end{eqnarray}
Here $\partial f_\alpha=\partial f(\frac{\partial\, \, }{\partial z^\alpha}), \bar{\partial}f_{\bar{\beta}}=\bar{\partial} f(\frac{\partial\, \, }{\partial \bar{z}^\beta})$, $R_{v\bar{v}}$ is viewed as transformation $T'M\to T'M$ defined as $R_{v\bar{v}}(\frac{\partial\, }{\partial z^\alpha})=R_{v\bar{v}\alpha\bar{\beta}}g^{\gamma \bar{\beta}}\frac{\partial}{\partial z^\gamma}$, and $\partial f(R^M_{v\bar{v}})_\alpha=\partial f(R^M_{v\bar{v}}(\frac{\partial\,\,}{\partial z^\alpha}))$.
\end{lemma}
\begin{proof} The proof is via direct computations by choosing normal coordinates centered at $p$ and $f(p)$. We can also derive this from the classical Kodaira-Bochner formula for $(1, 1)$-forms. (See for example \cite{MK} as well as Lemma 2.1 of \cite{NN}.) This is based on the following observation: Let $\eta$ denote the $(1,1)$-form  $f^* \omega_h$ with $\omega_h$ being the K\"ahler form of $N$. Then $\|\partial f\|^2$ is nothing but $\Lambda \eta$ (following the notation of \cite{NN} with $\Lambda$ being the contraction using the K\"ahler metric $\omega_g$). Hence the left hand side of the formula (\ref{Boch1}) amounts to computing $\partial \bar{\partial} \Lambda \eta$. On the other hand Lemma 2.1 of \cite{NT2} asserts that it equals to $\sqrt{-1}\Delta_{\bar{\partial}} \eta$ since $d\eta=0$. Now the Kodaira-Bochner formula (cf. Lemma 2.1 of \cite{NN}) can be applied to obtain the right hand side, since the Kodaira-Bochner formula expresses $\Delta_{\bar{\partial}} \eta$ in terms of curvature of $M$ together with  $\frac{1}{2}\left(\nabla_{\gamma}\nabla_{\bar{\gamma}}+ \nabla_{\bar{\gamma}}\nabla_\gamma\right) \eta$. Note that the first two terms in the right hand side of (\ref{Boch1}) comes from $\frac{1}{2}\left(\nabla_{\gamma}\nabla_{\bar{\gamma}}+ \nabla_{\bar{\gamma}}\nabla_\gamma\right) \eta$. There are also cancelations for terms involving $Ric^M$.
\end{proof}

\begin{corollary} \label{coro-to-lm1}(a) Let $M^m$ and  $N^n$ be  two complete K\"ahler manifolds.  Assume that the bisectional curvature of $M$ is nonnegative and the bisectional curvature of $N$ is non-positive. Let $f:M \to N$ be a holomorphic map. Then $\log (1+\|df\|^2)$ is plurisubharmonic. Moreover, if $\log (1+\|df\|^2)$ is pluriharmonic, then $f$ is totally geodesic.

(b) If $M^m$ is a Riemann surface with nonnegative curvature (whose universal cover is $\mathbb{C}$), the same result holds if $N$ has nonpositive holomorphic sectional curvature.
\end{corollary}

\begin{proof} Direct calculation shows that
$$
\langle \sqrt{-1}\partial \bar{\partial} \log\left(1+\|\partial f\|^2\right), \frac{1}{\sqrt{-1}}v\wedge \bar{v}\rangle =\frac{\partial_v \bar{\partial}_{\bar{v}}\|\partial f\|^2}{1+\|\partial f\|^2}-\frac{\left|\partial_v \|\partial f\|\right|^2}{\left(1+\|\partial f\|^2\right)^2}.
$$
Under the curvature assumption of (a) we have that, under the normal coordinates
\begin{eqnarray*}
\langle \sqrt{-1}\partial \bar{\partial} \log\left(1+\|\partial f\|^2\right), \frac{1}{\sqrt{-1}}v\wedge \bar{v}\rangle&\ge& \frac{\|D_v \partial_{(\cdot)} f\|^2}{1+\|\partial f\|^2}-\frac{\left|\partial_v \|\partial f\|\right|^2}{\left(1+\|\partial f\|^2\right)^2}\\
&=&\frac{\|D_v \partial_{(\cdot)} f\|^2}{(1+\|\partial f\|^2)^2}+\frac{ D_v f^i_{\alpha} D_{\bar{v}}f^{\bar{i}}_{\bar{\alpha}} |f^k_{\gamma}|^2 -\left| D_vf^i_{\alpha}f^{\bar{i}}_{\bar{\alpha}} \right|^2}{(1+\|\partial f\|^2)^2}\\
&\ge& \frac{\|D_v \partial_{(\cdot)} f\|^2}{(1+\|\partial f\|^2)^2}.
\end{eqnarray*}
From this estimate the claims in (a) follow easily.
If $m=1$, let $v=a\frac{\partial\, }{\partial z^1}$ for some $a\in \mathbb{C}$. Hence
$$R^N(\partial f_\alpha, \bar{\partial}f_{\bar{\beta}}, \partial f(v), \bar{\partial} f(\bar{v}))
=|a|^2R_{i\bar{i} j\bar{j}}|f^i_1|^2 |f^{j}_1|^2\le 0
$$
under the assumption that $N$ has nonnegative holomorphic sectional curvature \cite{Zheng-B}. The rest of the proof is the same.
\end{proof}

Another quantity which enjoys a similar Bochner type formula is $\frac{(f^* \omega_h)^m}{\omega_g^m}$ for the case $m\le n$. Here $\omega_h=\sqrt{-1}\sum_{i, j=1}^nh_{i\bar{j}}dw^i\wedge d\bar{w}^j$ is the K\"ahler form of $N^n$ and $\omega_g=\sqrt{-1} \sum_{\alpha, \beta =1}^m g_{\alpha\bar{\beta}}dz^\alpha\wedge d\bar{z}^\beta$ is the K\"ahler form of $M^m$. The equal dimensional case for the Lapalacian operator $\Delta$ (instead of $\partial\bar{\partial}$) was considered previously by various people, including Kobayashi, Yau, Mok-Yau, etc. We refer the readers to  \cite{Kobayashi-H}, \cite{Zheng-B} and references therein for details.

\begin{lemma}\label{lemma1-2}  For a holomorphic map $f: M^m \to N^n$ such that $df$ has rank  $m$ in a neighborhood of $p$, Let $D(x)=\frac{(f^* \omega_h)^m(x)}{\omega_g^m(x)}$ (which is positive in a neighborhood of $p$). Then for normal coordinates centered at $p$ and $f(p)$ such that at $p$, $df\left(\frac{\partial\, }{\partial z^{\alpha}}\right)=\lambda_\alpha \delta_{i\alpha} \frac{\partial\, }{\partial w^i}$ (namely \{$|\lambda_\alpha|^2$\} are singular values of $\partial f: T_p'M\to T_{f(p)}'N$), we have at $p$,
\begin{eqnarray}\label{Boch2}
\langle \sqrt{-1}\partial \bar{\partial} \log D, \frac{1}{\sqrt{-1}}v\wedge \bar{v}\rangle &=& \sum_{\alpha=1}^m \sum_{ m+1\le i  \le  n} \frac{|f^i_{\alpha v}|^2}{|\lambda_\alpha|^2}-\sum_{\alpha =1}^m R^N(\alpha, \bar{\alpha}, \partial f(v), \overline{\partial f(v)})   \nonumber\\
&\quad&  + Ric^M( v, \bar{v}).
\end{eqnarray}
Here $R^N(\alpha, \bar{\alpha}, \partial f(v), \overline{\partial f(v)})=R^N(\frac{\partial \, }{\partial w^{\alpha}}, \frac{\partial \, }{\partial \bar{w}^{\alpha}}, \partial f(v), \overline{\partial f(v)})$.
\end{lemma}
\begin{proof} As stated in the lemma, after unitary changes of frames of $T'_pM$ and $T_{f(p)}'N$, $df$, or $\partial f$ can be expressed as  $df\left(\frac{\partial\, }{\partial z^{\alpha}}\right)=\partial f\left(\frac{\partial\, }{\partial z^{\alpha}}\right)=\lambda_\alpha \delta_{i\alpha} \frac{\partial\, }{\partial w^i}$. We can perform the computation at $p$ and $f(p)$, where
$$
R^M_{\alpha \bar{\beta} \gamma \bar{\delta}}=-g_{\alpha \bar{\beta}, \gamma \bar{\delta}}, \quad R^N_{i\bar{j}k\bar{l}}=-h_{i\bar{j}, k\bar{l}}.
$$
Here $g_{\alpha \bar{\beta}, \gamma \bar{\delta}}=\frac{\partial^2 g_{\alpha\bar{\beta}}}{\partial z^\gamma \partial \bar{z}^\delta}$. Moreover at $f(p)$, $h_{i\bar{j}, k}=h_{i\bar{j}, \bar{k}}=0$.  To simplify  notations we write $\frac{\partial f^{i}}{\partial z^\alpha}$ as $f^{i}_\alpha$, and
$\frac{\partial^2 f^i}{\partial z^\alpha \partial z^{\gamma}}$ as $f^i_{\alpha \gamma}$. For $v=\sum_{\gamma} v^\gamma\frac{\partial \, }{\partial z^\gamma}$, $f^i_{\alpha v}=\sum_{\gamma} f^i_{\alpha \gamma}v^\gamma$. With respect to such coordinates let  $A=(A_{\alpha \bar{\beta}})$ be the Hermitian symmetric matrix with
$A_{\alpha \bar{\beta}}=f^i_{\alpha} h_{i\bar{j}}\overline{f^{j}_\beta}$. Then $D=\frac{\det(A)}{\det(g_{\alpha \bar{\beta}})}$. We denote $(A^{\alpha\bar{\beta}})$ as the inverse of $A$. Hence
$$
\frac{\partial^2\quad}{\partial z^{\gamma} \partial \bar{z}^\delta} \log D=\frac{\partial^2\quad}{\partial z^{\gamma} \partial \bar{z}^\delta} \log \det(A)+Ric^M(\gamma, \bar{\delta}).
$$
Direct calculation shows that
\begin{eqnarray*}
\left(\log \det(A)\right)_{\bar{\delta}}&=& A^{\alpha \bar{\beta}}\left[ f^i_\alpha h_{i\bar{j}}\overline{f^j_{\beta \delta}}+f^i_{\alpha} h_{i\bar{j}, \bar{l}} \overline{f^j_{\beta}}\overline{f^l_{ \delta}}\right]\\
&=&\sum_{\alpha} \frac{\overline{f^\alpha_{\alpha \delta}}\cdot  \lambda_\alpha}{|\lambda_\alpha|^2}.
\end{eqnarray*}
The last line only holds at the point $p$, while the first holds in the neighborhood.
Similarly,
\begin{eqnarray*}
\left(\log \det(A)\right)_{\gamma}&=& A^{\alpha \bar{\beta}}\left[ f^i_{\alpha \gamma} h_{i\bar{j}}\overline{f^j_{\beta}}+f^i_{\alpha} h_{i\bar{j}, k} \overline{f^j_{\beta}}f^k_{ \gamma}\right]\\
&=&\sum_{\alpha} \frac{f^\alpha_{\alpha \gamma}\cdot  \overline{\lambda_\alpha}}{|\lambda_\alpha|^2}.
\end{eqnarray*}
Taking second derivative, and at the end restricting to $p$,  we have
\begin{eqnarray*}
\left(\log \det(A)\right)_{\gamma\bar{\delta}}&=& -A^{\alpha \bar{s}}\frac{\partial A_{t\bar{s}}}{\partial z^\gamma} A^{t\bar{\beta}}\left[ f^i_\alpha h_{i\bar{j}}\overline{f^j_{\beta \delta}}+f^i_{\alpha} h_{i\bar{j}, \bar{l}} \overline{f^j_{\beta}}\overline{f^l_{ \delta}}\right] \\
&\quad&+\sum_{1\le \alpha\le m, 1\le i\le n}\frac{f^i_{\alpha \gamma} \overline{f^i_{\alpha \delta}}}{|\lambda_\alpha|^2}+\sum_{\alpha}\sum_{i j kl} \frac{-R^N_{i\bar{j}k\bar{l}}f^i_{\alpha} \overline{f^j_\alpha}f^k_{\gamma}\overline{f^l_\delta}}{|\lambda_\alpha|^2} \\
&=&\sum_{\alpha=1}^m \sum_{ m+1\le i  \le  n} \frac{f^i_{\alpha \gamma} \overline{f^i_{\alpha \delta}}}{|\lambda_\alpha|^2}-\sum_{\alpha=1}^mR^N_{\alpha\bar{\alpha} \gamma \bar{\delta}}\lambda_\gamma \overline{\lambda_\delta}.
\end{eqnarray*}
 Here we have used that
$$
A^{\alpha \bar{s}}\frac{\partial A_{t\bar{s}}}{\partial z^\gamma} A^{t\bar{\beta}}\left[ f^i_\alpha h_{i\bar{j}}\overline{f^j_{\beta \delta}}+f^i_{\alpha} h_{i\bar{j}, \bar{l}} \overline{f^j_{\beta}}\overline{f^l_{ \delta}}\right]=\sum_{\alpha, \beta=1}^m \frac{f^\alpha_{\beta \gamma} \overline{f^\alpha_\alpha} f^\alpha_\alpha \overline{f^\alpha_{\beta \delta}}}{|\lambda_\alpha|^2|\lambda_\beta|^2}=\sum_{\alpha, \beta} \frac{f^{\alpha}_{\beta\gamma}\overline{f^\alpha_{\beta \delta}}}{|\lambda_\beta|^2}.
$$
Putting all the above together we have the Bochner formula claimed.
\end{proof}

\begin{corollary}\label{coro-to-lm2} If  $Ric^N_m\le 0$  and $Ric^M \ge 0$,
$\log D$ is a plurisubharmonic function.
\end{corollary}

\section{ Proof of Theorem \ref{thm1} and Theorem \ref{thm2}}

To prove part (i) of Theorem \ref{thm1}, note that $u(x)=\log (\|\partial f\|^2(x)+1)$ is a plurisubharmonic function by part (a) of  Corollary \ref{coro-to-lm1}. The growth assumption of the gradient in the theorem implies that $u(x)=o\left(\log(r(x))\right)$. Hence Theorem 0.2 of \cite{NT1} implies that $u$ is a constant. This together with the second part of (a) in Corollary \ref{coro-to-lm1} implies that $f$ is totally geodesic.

To prove part (ii) of Theorem \ref{thm1}, we use part (b) of Corollary \ref{coro-to-lm1} instead.

Here we should remark that the argument of proving the  three circle theorem in \cite{Liu} works without any changes. Namely one can conclude the following corollary.

\begin{corollary}\label{coro:31} Let $M, N$ be as in Theorem \ref{thm1}. Let $f: M\to N$ be a holomorphic map.
Let $M(r)=\sup_{x\in B_p(r)}\|\partial f\|(x)$. Then for any $r_1<r_2<r_3$:
\begin{equation}\label{eq:31}
\log M(r_2)\le \frac{1}{\log r_3-\log r_1}\left( (\log r_3-\log r_2)\log M(r_1)+(\log r_2-\log r_1) \log M(r_3)\right).
\end{equation}
\end{corollary}
This together with the consequences of (\ref{eq:31}) derived  in \cite{Liu} implies that the boundedness of $\|\partial f\|$ follows from (\ref{eq:11}).

To prove part (i) of Theorem \ref{thm2}, we argue by contradiction. Assume that there exists a holomorphic map $f$ such that $\partial f: T'_xM \to T'_{f(x)}N$ is of full rank for some $x$. Let $D$ be the function defined in Lemma \ref{lemma1-2}. Since $M$ is compact, $D$ attains its maximum at some point $x_0$. Then in a neighborhood of $x_0$,  $D\ne 0$. Now let $\{|\lambda_\gamma|^2(x)\}$ be the singular value of $\partial f$ at $x$. Define the following second order elliptic operator pointwisely
$$
\mathcal{L}=\sum_{\gamma}\frac{1}{2|\lambda_{\gamma}|^2}\left(\nabla_{\gamma}\nabla_{\bar{\gamma}}
+\nabla_{\bar{\gamma}}\nabla_{\gamma}\right).
$$
Applying (\ref{Boch2}), we obtain that at $x_0$, with respect to the normal coordinates specified in Lemma \ref{lemma1-2}
$$
0\ge \mathcal{L} \log D\ge -\sum_{\gamma=1}^m \sum_{\alpha =1}^m \frac{1}{|\lambda_\gamma|^2}R^N\left(\alpha, \bar{\alpha}, \partial f\left(\frac{\partial\,}{\partial z^\gamma}\right), \overline{\partial f\left(\frac{\partial\,}{\partial z^\gamma}\right)}\right).
$$
Note that $\sum_{\gamma=1}^m \sum_{\alpha =1}^m \frac{1}{|\lambda_\gamma|^2}R^N(\alpha, \bar{\alpha}, \partial f(\frac{\partial\,}{\partial z^\gamma}), \overline{\partial f(\frac{\partial\,}{\partial z^\gamma})})$ is nothing but the $m$-scalar curvature of $\Sigma=\mbox{Span}\{ \partial f(\frac{\partial\,}{\partial z^\gamma})\}$, which is negative by the assumption of Theorem \ref{thm2} (namely the right hand above is positive).
This is a contradiction to $\mathcal{L}\log D\le 0$ at $x_0$. Similar argument proves the same result if $Ric^M>0$ and $S^N_m \le0$.

Note that the above argument  implies a rigidity result when $S_m^N \le 0$ is allowed (but with other assumptions).

\begin{corollary}\label{coro-to-coro22} Assume that $\dim_{\mathbb{C}}M=m \le n=\dim_{\mathbb{C}}N$. Let $(M, g)$ be a compact K\"ahler manifold such that $Ric^M \ge 0$. Let $(N^n, h)$ be a complete K\"ahler manifold such that $S^N_m(y)\le0$. Then for any non-degenerate holomorphic map $f:M\to N$,  $D$ must be a constant. Moreover $Ric^M \equiv 0$, and $S^N_m=0$ at least along a $m$-dimensional submanifold. If furthermore $Ric^N_m \le 0$ and $(M^m, g)$ has nonnegative bisectional curvature, then $f$ must be totally geodesic. Any holomorphic map $f: M \to N$ is totally geodesic if  $H^N\le 0$  and $M$ is compact with $Ric^M\ge 0$.
\end{corollary}
\begin{proof} Given that $f$ is holomorphic, the locus where $D\ne0$ is open and dense, with its complement being a closed subvariety. Over this open dense subset
$$
\mathcal{L} \log D \ge 0.
$$
Hence $\log D$ must be a constant since it attains an interior maximum.  The Ricci flat part and $S^N_m$ vanishing along a $m$-submanifold follow from Lemma \ref{lemma1-2}.

Under the  condition that $(M, g)$ has nonnegative bisectional curvature and $Ric^N_m \le 0$, since $D=constant$ on $M$, we apply the operator $\mathcal{L}$ (which is well defined due to that $D\ne 0$) to $\|\partial f\|^2$. Then  Lemma \ref{Boch1} implies that
$$
\mathcal{L} \|\partial f\|^2 =\sum_{i=1}^n \sum_{\alpha, \gamma=1}^m \frac{|f^i_{\alpha \gamma}|^2}{|\lambda_\gamma|^2}-\sum_{\alpha}Ric^N(\alpha, \overline{\alpha})|\lambda_\alpha|^2+\sum_{\alpha, \gamma} R^M_{\gamma \bar{\gamma} \alpha \bar{\alpha}}\frac{|\lambda_\alpha|^2}{|\lambda_\gamma|^2}\ge0.
$$
The maximum principle implies that $\mathcal{L}\|\partial f\|^2=0$ and $\|\partial f\|$ is a constant. The part $f$ being total geodesic follows from $f^{i}_{\alpha \gamma}=0$ for all $1\le i\le n, 1\le \alpha, \gamma \le m$.

For the case $H^N\le 0$ and $Ric^M \ge 0$, Lemma \ref{Boch1} implies that $$
\Delta \|\partial f\|^2 =\sum_{i=1}^n \sum_{\alpha, \gamma=1}^m |f^i_{\alpha \gamma}|^2-R^N_{i\bar{i}j\bar{j}}|f^i_\alpha|^2|f^j_\gamma|^2+Ric^M_{\alpha \bar{\alpha}}|f^i_\alpha|^2\ge 0.
$$
The argument in Appendix implies that the right hand side above is nonnegative. Then claim then follows by the maximum principle. \end{proof}

  Recall from the introduction that we say $N$ has $k$-dimensional Ricci curvature bounded from above by $-\kappa$ (denoted as $Ric_k^N(v, \bar{v})\le -\kappa |v|^2$), if when restricted to any $k$-dimensional subspace $\Sigma \subset T_y'N$, the Ricci curvature of curvature tensor of $R^N|_\Sigma$,
$$
Ric_{y, \Sigma}(v, \bar{v})\doteqdot \sum_{\gamma=1}^k R(E_\gamma, \overline{E}_\gamma, v, \bar{v})
$$
is bounded from above by $-\kappa |v|^2$ for any $v\in \Sigma$. Here $\{E_\gamma\}$ is a unitary basis of $\Sigma\subset T_y'N$. Note that for $k=1$, $Ric_1$ is the same as the holomorphic sectional curvature. However $Ric^N_k$ for $k\ge 2$ is independent of the holomorphic sectional curvature $H^N$ in view of the examples in \cite{Hitchin} and \cite{NZ}. The following Schwarz type estimate generalizes previous one proved for $m=n$ (cf. page 190 of  \cite{Zheng-B}).

\begin{corollary}\label{Coro-31} Let $f:M^m\to N^n$ ($m\le n$) be a holomorphic map with $M$ being a complete manifold. Assume that  $Ric^M $ is bounded from below and  the scalar curvature $S^M(x)\ge -K$.  Assume further that the $m$-Ricci  of $N$,  $Ric^N_m(x)\le -\kappa<0$. Then we have the  estimate
$$
D\le \left(\frac{K}{m\kappa}\right)^m.
$$
\end{corollary}
\begin{proof} Note that Lemma \ref{lemma1-2} implies that
$$
\Delta \log D \ge \kappa \sum_{\gamma} |\lambda_\gamma|^2 -K\ge mD^{\frac{1}{m}}\kappa -K.
$$
The claimed result follows from a similar argument as in the proof of classical  Schwarz Lemma (see Theorem 7.23 of \cite{Zheng-B}) by applying suitable cut-off techniques and the maximum principle as in \cite{CY} (see also the next section). The lower bound of the Ricci curvature is needed to apply the Laplacian comparison theorem on distance function (later a stronger lower bound is needed to apply the Hessian comparison theorem).
\end{proof}

 The part (ii) of Theorem \ref{thm2} is an immediate consequence of the above estimate applying to $K=0$. Note that the negative upper bound $Ric^N_m\le -\kappa$ holds if $N$ is compact with $Ric^N_m <0$. Similar argument to the proof of Corollary \ref{coro-to-coro22} implies the following result.

\begin{corollary}\label{coro-to-coro31} Assume that $\dim_{\mathbb{C}}M=m \le n=\dim_{\mathbb{C}}N$. Let $(M, g)$ be a compact K\"ahler manifold such that $S^M \ge 0$. Let $(N^n, h)$ be a complete K\"ahler manifold such that $Riv^N_m(y)\le0$. For any non-degenerate holomorphic map $f:M\to N$,  $D$ must be a constant. Moreover $S^M \equiv 0$, and $Ric^N_m=0$ at least along a $m$-dimensional submanifold.
\end{corollary}

By flipping the sign we  have the following consequences.

\begin{corollary}[Hoop Lemma-Volume]\label{coro-positive-2}Let $f:M\to N$ be a holomorphic map with $M$ being compact. Assume that  $Ric^M\ge K>0$.  Assume further that the $m$-Ricci of $N$,   $Ric^N_m(x)\le \kappa$ with $\kappa>0$. Then we have the  estimate
$$
\max_{x\in M} D^{1/m}(x)\ge \frac{K}{\kappa},
$$
provided that $f$  is non-degenerate.
\end{corollary}
\begin{proof} At the maximum point of $D$, say $x_0$, apply Lemma \ref{lemma1-2} as before. Pick $v$ to be the unit direction such that $|\partial f(v)|$ is the smallest. Then maximum principle implies that
$0\ge -\kappa \inf_{v, |v|=1} |\partial f(v)|^2 +K.$ The claimed result follows easily.
\end{proof}
We should remark that a similar result can be obtained (with the same argument) for the harmonic maps between two Riemannian manifolds. Namely if $u: M\to N$ is a harmonic map between two compact Riemannian manifolds. Assume that the sectional curvature of $N$ is bounded from above by $\kappa$ and $Ric^M\ge K$ with $\kappa, K>0$. Then for nonconstant map $u$
\begin{equation}\label{eq:hoop-harmonic}
\max_{x\in M} \|du\|^2 (x) \ge \frac{K}{\kappa}.
\end{equation}
 The corollary above has the advantage that when  the volume (or the stretching of the volume forms)  is concerned {\it only the Ricci curvatures} of both the target and domain manifolds are involved. The result for the harmonic maps is less satisfying since it involves bound of two different type of curvatures.

To prove part (iii) of Theorem \ref{thm2}, we need the following result from \cite{Ni-JDG}:

{\it Let $(M, g)$ be a complete K\"ahler manifold with $Ric^M \ge 0$. Let $u(x)$ be a plurisubharmonic function on $M$ satisfying that
$$
\lim_{x\to \infty} u(x)=o\left(\log(r(x))\right)
$$
then $\left(\sqrt{-1} \partial \bar{\partial} u\right)^m\equiv 0$.}

Now let $u(x)=\log D(x)$. Lemma \ref{lemma1-2} implies that $u(x)$ is a plurisubharmonic function, and at the point where $D>0$, it is strictly plurisubharmonic due to that $Ric_m^N<0$. On the other hand the growth assumption in part (iii) of Theorem \ref{thm2} implies that $u(x)=o(\log(r(x)))$. Hence $\left(\sqrt{-1}\partial \bar{\partial} u\right)^m\equiv 0$. This is a contradiction at the point $x$ with $D(x)>0$. The contradiction shows that $D(x)\equiv 0$, namely $f$ is degenerate.

\section{Extensions}

In this section we extend the proofs in the previous section to obtain the following result towards the question (Q) raised in the introduction.

\begin{theorem}\label{thm41} Assume that $\dim_{\mathbb{C}}M=m \le n=\dim_{\mathbb{C}}N$. (i) Let $(M, g)$ be a complete K\"ahler manifold such that the holomorphic  bisectional curvature  is bounded from below by  $-K_1$ for some $K_1>0$. Let $(N^n, h)$ be a compact K\"ahler manifold such that $S_m(y)\le -\kappa<0$.  Let $\{|\lambda_\gamma|^2(x)\}$ be the singular values of $\partial f$ at $T_x'M$. Assume further that $Ric^M \ge -K$ and $D$ is bounded from above, then
\begin{equation}\label{eq:41}
\limsup_{x\to \infty}(\min_{\gamma} |\lambda_\gamma|^2(x))\le \frac{mK}{\kappa}.
\end{equation}
In particular if additionally $Ric^M \ge 0$  then  $f:M\to N$ must be asymptotically degenerate in the sense that
$$
\limsup_{x\to \infty} (\min_{1\le\gamma\le m} |\lambda_\gamma|^2(x))=0.
$$

(ii) Let $(M, g)$ be a K\"ahler manifold. Assume that for $R>0$, the holomorphic bisectional curvature of $M$ is bounded from below by $-K_1$ for some $K_1>0$ in $B_p(R)$. Let $(N^n, h)$ be a compact K\"ahler manifold such that $S_m(y)<-\kappa<0$.  Let $f:M\to N$ be  a holomorphic map. Let $\{|\lambda_\gamma|^2(x)\}$ be the singular values of $\partial f$ at $T_x'M$. Let  $\sigma_{m-1}(\lambda)$ be the  $(m-1)$-th symmetric function  of the singular values $\{|\lambda_\gamma|^2\}$. Assume further that $Ric^M \ge -K$. Then we have
\begin{equation}\label{eq:42}
\sup_{B_p(\frac{R}{2})} D \le \frac{mK}{\kappa}\sup_{B_p(R)} \sigma_{m-1}+\left(\frac{C_1}{R^2}+\frac{C_1}{ R}\left(C(m)\left(\frac{1}{R}+\sqrt{K_1}\right)\right)\right) \frac{\sup_{B_p(R)}\sigma_{m-1}}{\kappa}.
\end{equation}
Here $C_1>0$ is an absolute constant.
If furthermore $(M^m, g)$ is complete and has nonnegative bisectional curvature and $\sigma_{m-1}$ satisfies  that
 $$
 \limsup_{x\to \infty} \frac{\sigma_{m-1}(x)}{r^2(x)}=0.
 $$
 Then $f$ must be degenerate.
\end{theorem}
\begin{proof} To prove part (i), we apply the maximum principle of \cite{Omori} at infinity.  By the virtue of \cite{Omori} we have a sequence of points $x_k\to \infty$ such that $\lim_{k\to \infty} D(x_k)\to \sup_M D$, which we may assume without the loss of generality being positive,  and
$$
\lim_{k\to \infty} \left. \langle \sqrt{-1}\partial \bar{\partial} \log D, \frac{1}{\sqrt{-1}}v\wedge \bar{v}\rangle\right|_{x_k} \le 0.
$$
Applying Lemma \ref{lemma1-2}, if  denoting the lower bound of the Ricci curvature (of $M$)  by $-K$  we have that
$$
\limsup_{k\to \infty} \left(\kappa -K\sum_{\gamma=1}^m \frac{1}{|\lambda_\gamma|^2(x_k)}\right) \le 0.
$$
This implies that
$$
\limsup_{k\to \infty}(\min_{\gamma} |\lambda_\gamma|^2(x_k))\le \frac{mK}{\kappa}.
$$
This proves (\ref{eq:41}), which implies the rest of part (i).

To prove part (ii), let $\eta(t):[0, +\infty)\to [0, 1]$ be a function supported in $[0, 1]$ with $\eta'=0$ on $[0, \frac{1}{2}]$, $\eta' \le 0$, $\frac{|\eta'|^2}{\eta}+(-\eta'')\le C_1$. The construction of such $\eta$ is elementary.
Let $\varphi_R(x)=\eta(\frac{r(x)}{R})$.  When the meaning is clear we omit subscript $R$ in $\varphi_R$. Clearly $D\cdot \varphi$ attains a maximum somewhere at $x_0$ in $B_p(R)$.
Now we apply $\mathcal{L}$ to  $\log (D\varphi)$  at the maximum point $x_0$ (where $D\cdot\varphi$ also attains its maximum).
The first derivatives vanish at $x_0$, which implies
$$\nabla_\gamma D(x_0) =-D(x_0)\left.\frac{\eta'(\frac{r(x)}{R})}{R\eta(\frac{r(x)}{R}) } \nabla_\gamma r(x)\right|_{x_0};\quad    \nabla_{\bar{\gamma}}D(x_0) =-D(x_0)\left.\frac{\eta'(\frac{r(x)}{R})}{R\eta(\frac{r(x)}{R}) } \nabla_{\bar{\gamma}} r(x)\right|_{x_0}.
$$
Applying Lemma \ref{lemma1-2},  we have that  at $x_0$ (where we may assume $D\varphi>0$),
\begin{eqnarray*}
0&\ge& \mathcal{L}\log (D\varphi) \ge \kappa-K\sum_{\gamma}\frac{1}{|\lambda_\gamma|^2}+\mathcal{L} \log \varphi \\
&=& \kappa-K\sum_{\gamma}\frac{1}{|\lambda_\gamma|^2}+\frac{\eta''}{R^2\varphi} \sum_{\gamma}\frac{|\nabla_\gamma r(x)|^2}{|\lambda_\gamma|^2}+\frac{\eta'}{2R \varphi}\sum_{\gamma}\frac{\nabla^2_{\gamma\bar{\gamma}}r(x)+\nabla^2_{\bar{\gamma}\gamma} r(x)}{|\lambda_\gamma|^2}\\
&\quad&-\frac{|\eta'|^2}{\varphi^2 R^2}\sum_{\gamma}\frac{|\nabla_\gamma r(x)|^2}{|\lambda_\gamma|^2}\\
&\ge& \kappa -K\sum_{\gamma}\frac{1}{|\lambda_\gamma|^2}-\frac{C_1}{\varphi R^2} \sum_{\gamma}\frac{1}{|\lambda_\gamma|^2}-\frac{C_1}{\varphi R} \sum_{\gamma}\frac{C(m)(\frac{1}{R}+\sqrt{K_1})}{|\lambda_\gamma|^2} -\frac{|\eta'|^2}{R^2 \eta \varphi}\sum_{\gamma}\frac{1}{|\lambda_\gamma|^2}.
\end{eqnarray*}
 In the last line above we have used the complex Hessian comparison theorem of \cite{LW}. Now multiplying $D\varphi$ on both side of the estimate above we have at $x_0$
$$
0\ge D\cdot \varphi \kappa -m \varphi K \sigma_{m-1}-\frac{C_1}{R^2}\sigma_{m-1}-\frac{C_1}{ R}\left(C(m)\left(\frac{1}{R}+\sqrt{K_1}\right)\right)\sigma_{m-1}.
$$
From this we have that
$$
\sup_{B_p(\frac{R}{2})} D \le \frac{mK}{\kappa}\sup_{B_p(R)} \sigma_{m-1}+\left(\frac{C_1}{R^2}+\frac{C_1}{ R}\left(C(m)\left(\frac{1}{R}+\sqrt{K_1}\right)\right)\right)\frac{\sup_{B_p(R)} \sigma_{m-1}}{\kappa}.
$$
This proves (\ref{eq:42}). In the above estimate,  letting $K=0$, then letting $R\to \infty$, noting that
$\lim_{R\to \infty} \frac{\sup_{B_p(R)} \sigma_{m-1}}{R^2}=0$,
 we have the rest of the claim in part (ii). Here we have used the complex Hessian comparison result assuming   the bisectional curvature lower bound \cite{LW}.
\end{proof}

It is clear from the proof that  part (i) of Theorem \ref{thm41} holds if $Ric\ge -K$ outside a compact domain, or even only
$$
\liminf_{x\to \infty} Ric^M(x)\ge -K.
$$
In particular if  $
\liminf_{x\to \infty} Ric^M(x)\ge 0
$, we  have
$$
\limsup_{x\to \infty} (\min_{\gamma} |\lambda_\gamma|^2(x))=0.
$$
For part (ii), if we choose $\eta$ carefully, the following estimate can be proved: If $Ric^M \ge -K$, the bisectional curvature is bounded from below by $-K_1$ outside $B_p(R_0)$ for some $R_0>0$ then
\begin{eqnarray}\label{eq:43}
\sup_{B_p(\frac{R}{2})\setminus B_p(R_0)} D &\le& \frac{mK}{\kappa}\sup_{B_p(R)} \sigma_{m-1}+\left(\frac{C_1}{R^2}+\frac{C_1}{ R}\left(C(m)\left(\frac{1}{R}+\sqrt{K_1}\right)\right)\right) \frac{\sup_{B_p(R)}\sigma_{m-1}}{\kappa}\nonumber\\
&\quad& +\frac{C_2}{R}\sup_{B_p(R_0)}\sigma_{m-1}.
\end{eqnarray}
Here $C_1$ is an absolute constant, and $C_2=C_2(R_0)$.

A similar localization procedure also implies the following estimate:
\begin{corollary}\label{Coro:42} Let $R>0$ be a constant. Assume that scalar curvature $S^M(x)\ge -K$, and $Ric^M \ge -K_1$ in $B_p(R)$,  and  that the $k$-Ricci of $N$  $Ric^N_k(x)\le -\kappa$, then we have the  estimate
\begin{equation}\label{eq:44}
\sup_{B_p(\frac{R}{2})} mD^{1/m}\le \frac{K}{\kappa}+\frac{1}{\kappa}\left(\frac{C_1}{R^2}+\frac{C_1}{ R}\left(C(m)\left(\frac{1}{R}+\sqrt{K_1}\right)\right)\right).
\end{equation}
Here $C_1$ is an absolute constant.
\end{corollary}

For any $p$ one may define the {\it lower Ricci curvature radius} (abbreviated as $r^{l}_{Ric}(p)$) as the biggest $R$ such that $Ric\ge -\frac{1}{R^2}$ in $B_p(R)$.
For  $R=r^{l}_{Ric}$ the estimate simplifies into the form:
\begin{equation}\label{eq:45}
\sup_{B_p(\frac{R}{2})} mD^{1/m}\le \frac{K}{\kappa}+\frac{1}{\kappa}\frac{C_1}{R^2}.
\end{equation}

\section{ A new Scharwz Lemma}
In this section we prove Theorem \ref{thm:sch}. We start with a linear algebraic lemma.

\begin{lemma}\label{lemma:51} Let $A$ be a Hermitian symmetric matric which is semi-positive. Let $G$ be positive  Hermitian symmetric matrix. We denote $(G^{\alpha \bar{\beta}})$ the inverse of $G$. Then for any $s$
\begin{equation}\label{eq:51}
\sup_{v\ne 0} \frac{\langle A(v), \bar{v}\rangle}{\langle G(v), \bar{v}\rangle} \ge \frac{G^{s\bar{\beta}} A_{\alpha \bar{\beta}}G^{\alpha \bar{s}}}{G^{s\bar{s}}}\ge \inf_{v\ne 0} \frac{\langle A(v), \bar{v}\rangle}{\langle G(v), \bar{v}\rangle}
\end{equation}
\end{lemma}
\begin{proof} By linear algebra, there exists $a$ and $g$ such that $A=\overline{a}^{t} \cdot a$ and $G=\overline{g}^{t} \cdot g$. The positivity of $G$ implies that $g$ is non-singular.
Let $\{E_\gamma\}$ be $\{\frac{\partial }{\partial z^\gamma}\}$. Now the middle term can be expressed as
$$
\frac{ \langle G^{-1}A G^{-1} (E_s), \overline{E_s}\rangle}{\langle G^{-1}(E_s), \overline{E}_s\rangle}=\frac{ \langle \bar{a}^t \cdot a \cdot g^{-1} \cdot (\bar{g}^t)^{-1} (E_s), \overline{g^{-1}\cdot  (\bar{g}^t)^{-1} (E_s)}\rangle}{\langle (\bar{g}^t)^{-1} (E_s), \overline{(\bar{g}^t)^{-1}(E_s)}\rangle}.
$$
Let $w=(\bar{g}^t)^{-1}(E_s)$ and  $v=g^{-1}(w)$ (which is clearly nonzero). Then the right hand side can be written as
$$
\frac{\langle \bar{a}^t \cdot a \cdot g^{-1}(w), \overline{g^{-1}(w)}\rangle}{|w|^2}=\frac{\langle a(v), \overline{a(v)}\rangle}{|g(v)|^2}=\frac{\langle A(v), \overline{v}\rangle}{\langle G(v), \overline{v}\rangle}.
$$
 Now the  claimed result becomes obvious.
\end{proof}

Now we prove Theorem \ref{thm:sch}.
 Let $\eta$ and $\varphi$ be the cut-off functions as in the last section. We consider $\|\partial f\|_m^2 \varphi$. It must attain a maximum somewhere in $B_p(R)$, say at $x_0$. Now we pick  normal coordinates $(z_1, \cdots, z_m)$ centered at  $x_0$,  and $(w_1, \cdots, w_n)$ centered at $f(x_0)$ as before. Let $A=(A_{\alpha\bar{\beta}})$ locally  with
$$
A_{\alpha \bar{\beta}}(x)=f^{i}_{\alpha}(x) h_{i\bar{j}}(f(x)) \overline{f^{j}_\beta}(x).
$$
By unitary changes of frame of  $T'_{x_0}M$ and $T'_{f(x_0)}N$ we can assume that $f^i_\alpha =\delta_{i\alpha} \lambda_\alpha$ at $x_0$. We may also assume that
$$
\|\partial f\|_m^2(x_0)= |\lambda_1|^2\ge |\lambda_2|^2 \ge \cdots\ge |\lambda_m|^2.
$$
Now let
$$
 W(x)=\frac{g^{1\bar{\beta}}(x) A_{\alpha \bar{\beta}}(x)g^{\alpha \bar{1}}(x)}{g^{1\bar{1}}(x)}.
$$
By the choice of the normal coordinates specified as above we have that $W(x_0)=|\lambda_1|^2=\|\partial f\|_m^2(x_0)$.
The above lemma implies that $W(x)\le \|\partial f\|^2_m (x)$ for $x$ in the neighborhood of $x_0$. Hence $W(x) \cdot \varphi (x)$ still attains a local maximum at $x_0$, which is the same as $\|\partial f\|^2_m \cdot \varphi$ at $x_0$. In the terminology of viscosity solutions, $W(x)$ serves a smooth barrier for $\|\partial f\|^2_m(x)$.
We shall apply the maximum principle to $\log (W(x) \cdot \varphi (x))$.  For that we need another $\partial \bar{\partial}$-lemma.

\begin{lemma}\label{lemma:52} Under the above notations, at $x_0$, or at any point with the normal coordinates specified as above,
\begin{equation}\label{eq:52}
\langle \sqrt{-1}\partial \bar{\partial} \log W, \frac{1}{\sqrt{-1}}v\wedge \bar{v}\rangle = R^M_{1\bar{1}v\bar{v}}-R^N_{1\bar{1} \partial f(v) \overline{\partial f(v)}}+\frac{\sum_{i\ne 1}|f^i_{1v}|^2}{W}.
\end{equation}
\end{lemma}
\begin{proof} We shall compute $\frac{\partial^2}{\partial z^\gamma\partial z^{\bar{\gamma}}} \log W$. Then the claimed result will follows from this by linear combinations. Under the normal coordinates as specified as above we have that
$$
\frac{\partial g^{\alpha \bar{\beta}}}{\partial z^\gamma}=-g^{\alpha \bar{\delta}} \frac{\partial g_{s\bar{\delta}}}{\partial z^\gamma}g^{s\bar{\beta}}; \quad \frac{\partial^2 g^{\alpha \bar{\beta}}}{\partial z^\gamma\partial z^{\bar{\gamma}}}=-g^{\alpha \bar{\delta}} \frac{\partial^2 g_{s\bar{\delta}}}{\partial z^\gamma \partial z^{\bar{\gamma}}}g^{s\bar{\beta}}=R_{\alpha\bar{\beta}\gamma\bar{\gamma}}.
$$
Hence we have at $x_0$, noting that $\frac{\partial g_{\alpha\bar{\beta}}}{\partial z^\gamma}=0$ and $\frac{\partial h_{i\bar{j}}}{\partial w^k}=0$,
\begin{eqnarray*}
\frac{\partial \log W}{\partial z^{\gamma}} &=&\frac{g^{1\bar{\beta}}f^i_{\alpha \gamma}h_{i\bar{j}}f^{\bar{j}}_{\bar{\beta}}g^{\alpha \bar{1}}}{W}=\frac{f^1_{1\gamma}f^{\bar{1}}_{\bar{1}}}{W};\quad \, \frac{\partial \log W}{\partial z^{\bar{\gamma}}}= \frac{f^{\bar{1}}_{\bar1\bar{\gamma}}f^{1}_{1}}{W}; \\
\frac{\partial^2 \log W}{\partial z^\gamma \partial z^{\bar{\gamma}}} &=& \frac{2\frac{\partial^2 g^{1\bar{\beta}}}{\partial z^\gamma \partial z^{\bar{\gamma}}}f^i_\alpha h_{i\bar{j}}f^{\bar{j}}_{\bar{\beta}}g^{\alpha \bar{1}} +g^{1\bar{\beta}}f^i_\alpha f^{\bar{j}}_{\bar{\beta}} \frac{\partial^2 h_{i\bar{j}}}{\partial w^k\partial w^{\bar{l}}} f^k_{\gamma} f^{\bar{l}}_{\bar{\gamma}}g^{\alpha \bar{1}} +g^{1\bar{\beta}}f^i_{\alpha \gamma}h_{i\bar{j}}f^{\bar{j}}_{\bar{\beta}\bar{\gamma}}g^{\alpha \bar{1}} }{W}\\
&\quad&-\frac{|f^1_{1\gamma}|^2 |f^1_1|^2}{W^2} -R^M_{1\bar{1}\gamma\bar{\gamma}}.
\end{eqnarray*}
The claimed result follows by observing that $W=|f^1_1|^2$, and putting the above computations together. \end{proof}

Now with the above lemma we continue along  the same line of argument of the proof of Theorem \ref{thm41} and obtain at $x_0$ where $W\cdot \varphi$ attains its maximum:
\begin{eqnarray*}
0&\ge& \frac{\partial^2}{\partial z^1 \partial z^{\bar{1}}}\, \left(\log (W\varphi)\right) \ge R^M_{1\bar{1}1\bar{1}}-R^N_{1\bar{1}1\bar{1}}|f^1_1|^2 +  \frac{\partial^2 \log \varphi}{\partial z^1 \partial z^{\bar{1}}} \\
&\ge& -K +\kappa |f^1_1|^2+\frac{\eta''}{R^2\varphi} |\nabla_1 r(x)|^2+\frac{\eta'}{2R \varphi}\left(\nabla^2_{1\bar{1}}r(x)+\nabla^2_{\bar{1}1} r(x)\right)-\frac{|\eta'|^2}{\varphi^2 R^2}\cdot |\nabla_1 r(x)|^2\\
&\ge&  -K +\kappa |f^1_1|^2 -\frac{C_1}{\varphi R^2} -\frac{C_1}{\varphi R} \cdot C(m)\left(\frac{1}{R}+\sqrt{K_2}\right) -\frac{C_1|\eta'|^2}{R^2 \eta \varphi}.
\end{eqnarray*}
Here we applied the complex Hessian comparison theorem of \cite{LW} with $K_2$ being the lower bound of the bisectional curvature in $B_p(R)$. Multiplying $\varphi$ on the both sides we will have at $x_0$ the estimate
$$
\kappa |f^1_1|^2\varphi\le K\varphi+\frac{C_1}{R^2} +\frac{C_1}{ R} \cdot C(m)\left(\frac{1}{R}+\sqrt{K_2}\right) +\frac{C_1|\eta'|^2}{R^2 \eta }.
$$
Hence we arrive at the estimate
\begin{equation}\label{eq:53}
\sup_{B_p(\frac{R}{2})} \|\partial f\|_m^2 \le \frac{1}{\kappa}\left(K+\frac{C_1}{R^2}+\frac{C_1}{ R} \cdot C(m)\left(\frac{1}{R}+\sqrt{K_2}\right)\right).
\end{equation}
The claimed estimate in Theorem \ref{thm:sch} follows by letting $R\to \infty$.
The last statement on the holomorphic map being constant map follows easily by applying the estimate to the case $K=0$.

\begin{corollary}\label{Coro:51} Let $R>0$ be a constant such that the bisectional curvature of $M$ is bounded from below on $B_p(R)$ by $-K_2$. Assume that $H^M(X)\ge -K|X|^4$,   and  that    $H^N(Y)\le -\kappa |Y|^4$, then we have the  estimate
\begin{equation}\label{eq:54}
\sup_{B_p(\frac{R}{2})}  \|\partial f\|_m^2 \le \frac{1}{\kappa}\left(K+\frac{C_1}{R^2}+\frac{C_1}{ R} \cdot C(m)\left(\frac{1}{R}+\sqrt{K_2}\right)\right).
\end{equation}
Here $C_1$ is an absolute constant.
\end{corollary}

For any point $p$, we can similarly define the {\it lower bisectional curvature radius} being the biggest  $R$ such that the bisectional curvature is bounded by $-\frac{1}{R^2}$ on $B_p(R)$. Such radius is denoted as $r^{l}_{B}(p)$. Clearly if the bisectional curvature is nonnegative $r^l_B(p)=\infty$. For $R= r^{l}_{B}(p)  $ the above estimate has the simple form:
\begin{equation}\label{eq:54}
\sup_{B_p(\frac{R}{2})} \|\partial f\|_m^2 \le \frac{1}{\kappa}\left(K+\frac{C_1}{R^2}\right).
\end{equation}

A consequence from the proof also implies the following result which can be interesting in the  study of  holomorphic (even memomorphic) maps between compact K\"ahler manifolds.
\begin{corollary}[Hoop Lemma-Streching]\label{coro-hoop2}
 (i) Assume that $M$ is compact, $H^M(X)\ge K|X|^4$,   and     $H^N(Y)\le \kappa |Y|^4$, with $K, \kappa>0$. Then for any nonconstant $f: M\to N$
 $$
 \max_x\|\partial f\|^2_m(x)\ge \frac{K}{\kappa}.
 $$
 (ii) Assume that $M$ is compact, $Ric^M\ge K_1$, and that $H^N(Y)\le \kappa |Y|^4$, with $K_1,\kappa>0$. Then for any non-constant holomorphic map $f:M\to N$
 $$
 \max_{x}\|\partial f\|^2(x) \ge \frac{K_1}{\kappa}.
 $$
\end{corollary}
The proof of the second statement uses an estimate modifying (\ref{eq:a1}) in the appendix. The part (i) of the result is more satisfying since it only involves the holomorphic sectional curvature of both the target and domain manifolds.

The proof of the Schwarz Lemma  implies the following result.

\begin{corollary} Let $(M^m, g)$ be a compact K\"ahler manifold, and $(N^n, h)$ be a another K\"ahler manifold. Assume either that  $H^M(X)>0$ and $H^N(Y)\le 0$, or $H^M(X)\ge 0$ and $H^N(Y)<0$. Then any holomorphic map $f:M\to N$ must be a constant.
\end{corollary}
\begin{proof} Assume not, then $\|\partial f\|^2_m (x)$ attains a nonzero maximum somewhere, say at $x_0$. Applying the above proof of the Schwarz Lemma, at $x_0$, we have
$$
0\ge \frac{\partial^2 \log W}{\partial z^1 \partial z^{\bar{1}}} \ge R^M_{1\bar{1}1\bar{1}}-R^N_{1\bar{1}1\bar{1}}\|\partial f\|_m^2(x_0)>0.
$$
This contradiction proves the  result.
\end{proof}

Note that for the first case, namely under the assumptions $H^M(X)>0$ and $H^N(Y)\le 0$, the result also follows from the above Hoop Lemma part (i) by taking $\kappa\to 0$. This part was also proved independently in \cite{Yang} using a different method.

\section*{Appendix}
First we include an alternate algebraic part of the proof, by Royden, of the ``classical" Schwarz Lemma \cite{Roy}:

{\it Let $f: M^m\to N^n$ be a holomorphic map. Assume that the holomorphic sectional curvature of $N$, $H(Y)\le -\kappa |X|^4$ and the Ricci curvature of $M$, $Ric(X, \overline{X})\ge -K |X|^2$ with $\kappa, K>0$. Then
$$
\|\partial f\|^2 \le \frac{2d}{d+1}\frac{K}{\kappa}.
$$
Here $d=rank(f)$.}

The argument, which is due to F.  Zheng,  proves a lemma  of Royden.

The  estimate $\|\partial f\|^2\le\frac{K}{\kappa}$ was proved (by S.-T. Yau \cite{Yau-sch}) either for $M$ being a Riemann surface, or for the case $m\ge 2$ assuming that the bisectional curvature of $N$ is bounded from above by $-\kappa$ (cf. \cite{Yau-sch}). The above result of Royden covers Yau's estimate for the Riemann surfaces case while allowing weaker holomorphic sectional curvature upper bound on the target manifolds for any dimension of the domain manifolds. The algebraic ingredient is  needed in showing, under the assumption that $H^N(X)\le -\kappa |X|^4$, 
\begin{equation}\label{eq:B-a}
\Delta \|\partial f\|^2 \ge \frac{d+1}{2d}\kappa \|\partial f\|^4 -K\|\partial f\|^2.
\end{equation}

By  taking trace in (\ref{Boch1}) of Lemma \ref{lemma1-1} we have that
$$
\Delta \|\partial f\|^2 \ge -g^{\alpha \bar{\beta}}g^{\gamma \bar{\delta}}R^N (\partial f_\alpha,  \bar{\partial}f_{\bar{\beta}}, \partial f_\gamma, \bar{\partial}f_{\bar{\delta}})+g^{\alpha \bar{\beta}}g^{\gamma\bar{\delta}} \langle \partial f(R^M_{\gamma \bar{\delta}})_{\alpha}, \bar{\partial}f_{\bar{\beta}}\rangle.
$$
With respect to  normal coordinates chosen before (so that $f^i_\alpha =\delta_{i\alpha} \lambda_i$) the above can be written as
$$
\Delta \|\partial f\|^2 \ge-R^N_{i\bar{i}j\bar{j}}|f^i_\alpha|^2|f^j_\gamma|^2+Ric^M_{\alpha \bar{\alpha}}|f^i_\alpha|^2\ge -R^N_{i\bar{i}j\bar{j}}|f^i_\alpha|^2|f^j_\gamma|^2 -K \|\partial f\|^2.
$$
Thus the result follows easily after the point-wise estimate (under the assumption $H(X)\le -\kappa |X|^4$):
\begin{equation}\label{eq:a1}
R^N_{i\bar{i}j\bar{j}}|f^i_\alpha|^2|f^j_\gamma|^2\le -\frac{d+1}{2d}\kappa \|\partial f\|^4.
\end{equation}
To prove this, consider the vector $Y=\sum_{\lambda_i\ne 0} w^i\lambda_i \frac{\partial}{\partial w^i}$ (if $m\le n$, $Y\in \partial f(T'_xM)$). Direct calculation shows that
$$
\aint_{\mathbb{S}^{2d-1}} R^N(Y, \overline{Y}, Y, \overline{Y})\, d\theta(w)=\frac{2}{d(d+1)}R^N_{i\bar{i}j\bar{j}}|\lambda_i|^2|\lambda_j|^2
=\frac{2}{d(d+1)}R^N_{i\bar{i}j\bar{j}}|f^i_\alpha|^2|f^j_\gamma|^2.
$$
On the other hand
\begin{eqnarray*}
\aint_{\mathbb{S}^{2d-1}} R^N(Y, \overline{Y}, Y, \overline{Y})\, d\theta(w)&\le& -\kappa \aint_{\mathbb{S}^{2d-1}}|Y|^4\, d\theta(w)\\
&=& \frac{-\kappa}{d(d+1)} \left(2\sum |\lambda_i|^4+\sum_{i\ne j} |\lambda_i|^2\, |\lambda_j|^2\right)\\
&=&-\frac{\kappa}{d(d+1)}\left( \|\partial f\|^4+\sum |\lambda_i|^4\right)\\
&\le&-\frac{\kappa}{d(d+1)}\frac{d+1}{d} \|\partial f\|^4.
\end{eqnarray*}
Putting together we have (\ref{eq:a1}).

Secondly we prove that the  $k$-hyperbolicity defined in the introduction is the same as: For any $x\in N$, $v_1\wedge\cdots \wedge v_k\ne 0$ with $v_i\in T'_xN$, the pseudo norm
$$
\|v_1\wedge\cdots\wedge v_k\|_k\doteqdot \inf_{ df(\tilde{v}_1\wedge\cdots\wedge \tilde{v}_k)=v_1\wedge\cdots\wedge v_k} \|\tilde{v}_1\wedge\cdots\wedge \tilde{v}_k\|_{p}
$$
 with  $f:\mathbb{D}^k(1)\to N$ being holomorphic and  $f(0)=x$, is a norm. Here $\|\cdot\|_p$ denotes the Poincar\'e metric on $\mathbb{D}^k$ (naturally extended to $\wedge_k T'N$). The proof is parallel to the $k=1$ case. It is easy to show that if there exists a nondegenerate $f: \mathbb{C}^k \to N$, then the above pseudo-norm has to vanish for any $v_1\wedge\cdots v_k\ne 0$  in $\wedge_k T_x'N$  with $df(\tilde{v}_1\wedge\cdots\wedge \tilde{v}_k)=v_1\wedge\cdots\wedge v_k$, since if we let $\rho_r(z)=\frac{z}{r}$, and define $f_\ell=f(\rho_{\frac{1}{\ell}}(z))$, it is clear  $df_\ell(\frac{\tilde{v}_1}{\ell}\wedge \cdots \wedge \frac{\tilde{v}_k}{\ell})=v_1\wedge\cdots\wedge v_k$. But $\|\frac{\tilde{v}_1}{\ell}\wedge \cdots \wedge \frac{\tilde{v}_k}{\ell}\|_p \to 0$ as $\ell \to \infty$. The other direction utilizes that $N$ is compact. It suffices to show  that if the pseudo-norm is not a norm at $x$, then one can construct a nondegenerate holomorphic $f: \mathbb{C}^k\to N$.  First equip $N$ with a Hermitian metric $h$, and consider $\mathcal{F}=\{ \mbox{ holomorphic } g: \mathbb{D}^k(1)\to N, g(0)=x\}$. We claim that there exists $g_i$ such that $D(g_i)(0)\to \infty$. Otherwise there will be a uniform upper bound $A$ for $D(g)$ for any $g\in \mathcal{F}$. This would implies that
$$
0<\|v_1\wedge\cdots\wedge v_k\|_h\le \sqrt{A} \|\tilde{v}_1\wedge\cdots\wedge \tilde{v}_k\|_{p}.
$$
A contradiction! Now for the $g_i$ with $D(g_i)(0)\to \infty$ we let $\ell_i=D(g_i)(0)$ and consider $f_i=g_i(\rho_{\ell_i}(z))$, which shall be defined on sequence of balls whose union covers the whole $\mathbb{C}^k$. Clearly $D(f_i)(0)=1$. Restricted to any compact subset $K\subset \mathbb{C}^k$, by the compactness of $N$ and passing to a subsequence (still denoted as $\{f_i\}$) we assume  $f_i\to f_\infty$ for some $f_\infty: \mathbb{C}^k\to N$. Clearly $D(f_\infty)(0)=1$, hence non-degenerate.

\section*{Acknowledgments} { We would like to thank,  Bernard Shiffman  for the reference \cite{Eis}, and him and Min Ru for their interest in the results related to the  $k$-hyperbolicity, Fangyang Zheng for pointing out a discrepancy in an  earlier version,  his proof in the Appendix and many discussions.}

\end{document}